\documentclass[11pt]{article}
\usepackage{amsmath,amssymb,amsfonts,amsthm}
\usepackage{float}
\numberwithin{equation}{section}
\newtheorem{theorem}{Theorem}[section]
\newtheorem{definition}{Definition}[section]

\newtheorem{proposition}[theorem]{Proposition}

\begin{document}
\begin{center}
{\Large{\textbf{{A note on $f^\pm$-Zagreb indices in respect of Jaco Graphs, $J_n(1), n \in \Bbb N$ and the introduction of Khazamula irregularity}}}} 
\end{center}
\vspace{0.5cm}
\large{\centerline{(Johan Kok, Vivian Mukungunugwa)\footnote {\textbf {Affiliation of author:}\\
\noindent Johan Kok (Tshwane Metropolitan Police Department), City of Tshwane, Republic of South Africa\\
e-mail: kokkiek2@tshwane.gov.za\\ \\
\noindent Vivian Mukungunugwa (Department of Mathematics and Applied Mathematics, University of Zimbabwe), City of Harare, Republic of Zimbabwe\\
e-mail: vivianm@maths.uz.ac.zw\\ \\
**On advice from arXiv Moderation this paper now incorporates similar ideas and variant results of another submission which has been removed.}}
\vspace{0.5cm}
\begin{abstract}
\noindent The topological graph indices $irr(G)$  related to the \emph{first Zagreb index}, $M_1(G)$ and the \emph{second Zagreb index}, $M_2(G)$ are of the oldest irregularity measures researched. Alberton $[3]$ introduced  the \emph{irregularity} of $G$ as $irr(G) = \sum \limits _{e\in E(G)} imb(e), imb(e) = |d(v) - d(u)|_{e=vu}.$ In the paper of Fath-Tabar $[7],$ Alberton's indice was named the \emph{third Zagreb indice} to conform with the terminology of chemical graph theory. Recently Ado et. al. $[1]$ introduced the topological indice called \emph{total irregularity}. The latter could be called the \emph{fourth Zagreb indice}. We  define the $\pm$\emph{Fibonacci weight}, $f_i^\pm$ of a vertex $v_i$ to be $-f_{d(v_i)},$ if $d(v_i)$ is uneven and, $f_{d(v_i)},$ if $d(v_i)$ is even.  From the aforesaid we define the $f^\pm$-Zagreb indices. This paper presents introductory results for the undirected underlying graphs of Jaco Graphs, $J_n(1), n\leq 12.$ For more on Jaco Graphs $J_n(1)$ see $[9, 10]$. Finally we introduce the \emph{Khazamula irregularity} as a new topological variant.\\ \\
We also present five open problems.
\end{abstract}
\noindent {\footnotesize \textbf{Keywords:} Total irregularity, Irregularity, Imbalance, Zagreb indices,  $\pm$Fibonacci weight, Total $f$-irregularity, Fibonaccian irregularity, $f^\pm$-Zagreb indices, Jaco graphs, Zeckendorf representation, Khazamula irregularity, Khazamula theorem.}\\ \\
\noindent {\footnotesize \textbf{AMS Classification Numbers:} 05C07, 05C20, 05C38, 05C75, 05C85} 
\section{Introduction} 
The topological graph indices $irr(G)$  related to the \emph{first Zagreb index}, $M_1(G) = \sum \limits _{v\in V(G)}d^2(v) = \sum \limits_{vu \in E(G)} (d(v) + d(u)),$ and the \emph{second Zagreb index}, $M_2(G) = \sum \limits_{vu \in E(G)}d(v)d(u)$ are of the oldest irregularity measures researched. Alberton $[3]$ introduced  the \emph{irregularity} of $G$ as $irr(G) = \sum \limits _{e\in E(G)} imb(e), imb(e) = |d(v) - d(u)|_{e=vu}.$ In the paper of Fath-Tabar $[7],$ Alberton's indice was named the \emph{third Zagreb indice} to conform with the terminology of chemical graph theory. Recently Ado et. al. $[1]$ introduced the topological indice called \emph{total irregularity} and defined it, $irr_t(G) =\frac{1}{2}\sum \limits_{u,v \in V(G)}|d(u) - d(v)|.$ The latter could be called the \emph{fourth Zagreb indice}. \\ \\
If the vertices of a simple undirected graph $G$ on $n$ vertices are labelled $v_i, i = 1, 2, 3, ..., n$ then the respective definitions may be:\\
$M_1(G) = \sum\limits_{i=1}^{n}d^2(v_i) = \sum \limits_{i=1}^{n-1}\sum\limits_{j=2}^{n} (d(v_i) + d(v_j))_{v_iu_j \in E(G)}, M_2(G) = \sum \limits_{i=1}^{n-1}\sum\limits_{j=2}^{n} d(v_i)d(v_j)_{v_iu_j \in E(G)},\\
M_3(G) = \sum \limits_{i=1}^{n-1}\sum\limits_{j=2}^{n} |d(v_i) - d(v_j)|_{v_iu_j \in E(G)}$ and $M_4(G) = irr_t(G) = \frac{1}{2}\sum \limits_{i=1}^{n} \sum \limits_{j=1}^{n}|d(v_i) - d(v_j)| = \sum \limits_{i=1}^{n} \sum \limits_{j=i+1}^{n}|d(v_i) - d(v_j)|$ or $\sum \limits_{i=1}^{n-1} \sum \limits_{j=i+1}^{n}|d(v_i) - d(v_j)|.$ For a simple graph on a singular vertex (\emph{1-empty} graph), we define $M_1(G) = M_2(G) = M_3(G)=M_4(G) = 0$ .
\section{Zagreb indices in respect of $\pm$Fibonacci weights, $f^\pm$-Zagreb indices}
 We  define the $\pm$\emph{Fibonacci weight}, $f_i^\pm$ of a vertex $v_i$ to be $-f_{d(v_i)},$ if $d(v_i) = i$ is uneven and, $f_{d(v_i)},$ if $d(v_i)$ is even. The $f^\pm$-Zagreb indices can now be defined as:\\ \\
$f^\pm Z_1(G) = \sum\limits_{i=1}^{n}(f^\pm_i)^2 = \sum \limits_{i=1}^{n-1}\sum\limits_{j=2}^{n} (|f^\pm_i| + |f^\pm_j|)_{v_iu_j \in E(G)}, f^\pm Z_2(G) = \sum \limits_{i=1}^{n-1}\sum\limits_{j=2}^{n} (f^\pm_i.f^\pm_j)_{v_iu_j \in E(G)},\\ \\
f^\pm Z_3(G) = \sum \limits_{i=1}^{n-1}\sum\limits_{j=2}^{n} |f^\pm_i - f^\pm_j|_{v_iu_j \in E(G)}$ and $f^\pm Z_4(G) = \frac{1}{2}\sum \limits_{i=1}^{n} \sum \limits_{j=1}^{n}|f^\pm_i - f^\pm_j| = \sum \limits_{i=1}^{n} \sum \limits_{j=i+1}^{n}|f^\pm_i - f^\pm_j|$ or $\sum \limits_{i=1}^{n-1} \sum \limits_{j=i+1}^{n}|f^\pm_i - f^\pm_j|.$ For a simple graph on a singular vertex (\emph{1-empty} graph), we define \\ \\$f^\pm Z_1(G) = f^\pm Z_2(G) = f^\pm Z_3(G)= f^\pm Z_4(G) = 0$ .
\subsection{Application to Jaco Graphs, $J_n(1), n \in \Bbb N$}
For ease of reference some definitions in $[9]$ are repeated. A particular family of finite directed graphs (\emph{order 1}) called Jaco Graphs and denoted by $J_n(1), n \in \Bbb N$ are directed graphs derived from a particular well-defined infinite directed graph (\emph{order 1}), called the \emph{1}-root digraph. The \emph{1}-root digraph has four fundamental properties which are; $V(J_\infty(1)) = \{v_i|i \in \Bbb N\}$ and, if $v_j$ is the head of an edge (arc) then the tail is always a vertex $v_i, i<j$ and, if $v_k,$ for smallest $k \in \Bbb N$ is a tail vertex then all vertices $v_ \ell, k< \ell<j$ are tails of arcs to $v_j$ and finally, the degree of vertex $k$ is $d(v_k) = k.$ The family of finite directed graphs are those limited to $n \in \Bbb N$ vertices by lobbing off all vertices (and edges arcing to vertices) $v_t, t > n.$ Hence, trivially we have $d(v_i) \leq i$ for $i \in \Bbb N.$
\begin{definition}
The infinite Jaco Graph $J_\infty(1)$ is defined by $V(J_\infty(1)) = \{v_i| i \in \Bbb N\}$, $E(J_\infty(1)) \subseteq \{(v_i, v_j)| i, j \in \Bbb N, i< j\}$ and $(v_i,v_ j) \in E(J_\infty(1))$ if and only if $2i - d^-(v_i) \geq j,$ $[9].$
\end{definition}
\begin{definition}
The family of finite Jaco Graphs are defined by $\{J_n(1) \subseteq J_\infty(1)|n\in \Bbb {N}\}.$ A member of the family is referred to as the Jaco Graph, $J_n(1),$ $[9].$
\end{definition}
\begin{definition}
The set of vertices attaining degree $\Delta (J_n(1))$ is called the Jaconian vertices of the Jaco Graph $J_n(1),$ and denoted, $\Bbb{J}(J_n(1))$ or, $\Bbb{J}_n(1)$ for brevity, $[9].$
\end{definition}
\noindent From $[9]$ we have \emph{Bettina's Theorem}.
\begin{theorem}
Let $\Bbb{F} = \{f_0, f_1,f_2,f_3, ...\}$ be the set of Fibonacci numbers and let $n=f_{i_1} + f_{i_2} + ... + f_{i_r}, n\in \Bbb N$ be the Zeckendorf representation of $n.$ Then 
\begin{center}
$d^+(v_n) = f_{i_1-1} + f_{i_2-1} + ... +f_{i_r-1}.$
\end{center}
\end{theorem}
\noindent Note: the degree of vertex $v_i$, denoted $d(v_i)$ refers to the degree in $J_\infty(1)$ hence $d(v_i) = i.$ In the finite Jaco Graph the degree of vertex $v_i$ is denoted $d(v_i)_{J_n(1)}.$ The degree sequence is denoted $\Bbb D_n = (d(v_1)_{J_n(1)}, d(v_2)_{J_n(1)}, ..., d(v_n)_{J_n(1)}).$ By convention $\Bbb D_{i+1} = \Bbb D_i \cup d(v_{i+1})_{J_n(1)}.$\\ \\
\noindent \textbf{2.1.1 Algorithm to determine the degree sequence of a finite Jaco Graph, $J_n(1), n \in \Bbb N.$}\\ \\
Consider a finite Jaco Graph $J_n(1), n \in \Bbb N$ and label the vertices $v_1, v_2, v_3, ..., v_n.$\\ \\
\noindent Step 0: Set $n = n.$ Let $i = j = 1.$ If $j = n = 1,$ let $\Bbb D_i = (0)$ and go to Step 6, else set $\Bbb D_i = \emptyset$ and go to Step 1.\\
Step 1: Determine the $j^{th}$ Zeckendorf representation say, $j = f_{i_1} + f_{i_2} + ... + f_{i_r},$ and go to Step 2.\\
Step 2: Calculate $d^+(v_j) = f_{i_1-1} + f_{i_2-1} + ... +f_{i_r-1},$ then go to Step 3.\\
Step 3: Calculate $d^-(v_j) = j - d^+(v_j),$ and let $d(v_j) = d^+(v_j) + d^-(v_j),$ then go to Step 4.\\
Step 4: If $d(v_j) \leq n$, set $d(v_j)_{J_n(1)} = d(v_j)$ else, set $d(v_j)_{J_n(1)} = d^-(v_j) + (n-j)$ and set $\Bbb D_j = \Bbb D_i \cup d(v_j)_{J_n(1)}$ and go to Step 5.\\
Step 5: If $j = n$ go to Step 6 else, set $i= i + 1$ and $j = i$ and go to Step 1.\\
Step 6: Exit.\\ \\
\noindent \textbf{2.1.2 Tabled values of $\Bbb F^\pm(J_n(1)),$ for finite Jaco Graphs, $J_n(1), n \leq 12.$}\\ \\
\noindent For illustration the adapted table below follows from the Fisher Algorithm $[9]$ for $J_n(1), n\leq 12.$ Note that the Fisher Algorithm determines $d^+(v_i)$ on the assumption that the Jaco Graph is always sufficiently large, so at least $J_n(1), n \geq i+ d^+(v_i).$ For a smaller graph the degree of vertex $v_i$ is given by $d(v_i)_{J_n(1)} = d^-(v_i) + (n-i).$ In $[9]$ Bettina's theorem describes an arguably, closed formula to determine $d^+(v_i)$. Since $d^-(v_i) = n - d^+(v_i)$ it is then easy to determine $d(v_i)_{J_n(1)}$ in a smaller graph $J_n(1), n< i + d^+(v_i).$ The $f^\pm_i$-sequence of $J_n(1)$ is denoted $\Bbb F^\pm(J_n(1)).$\\ \\ \\ \\ \\ \\ \\ \\ \\ \\ \\ \\ \\ \\
\noindent Table 1. \\
\begin{tabular}{|c|c|c|c|}
\hline
$i\in{\Bbb{N}}$&$d^-(v_i)$&$d^+(v_i) = i - d^-(v_n)$&$\Bbb F^\pm(J_i(1))$\\
\hline
1&0&1&(0)\\
\hline
2&1&1&(-1, -1)\\
\hline
3&1&2&(-1, 1, -1)\\
\hline
4&1&3&(-1, 1, 1, -1)\\
\hline
5&2&3&(-1, 1, -2, 1, 1)\\
\hline
6&2&4&(-1, 1, -2, -2, -2, 1)\\
\hline
7&3&4&(-1, 1, -2, 3, 3, -2, -2)\\
\hline
8&3&5&(-1, 1, -2, 3, -5, 3, 3, -2)\\
\hline
9&3&6&(-1, 1, -2, 3, -5, -5, -5, 3, -2)\\
\hline
10&4&6&(-1, 1, -2, 3, -5, 8, 8, -5, 3, 3)\\
\hline
11&4&7&(-1, 1, -2, 3, -5, 8, -13, 8, -5, -5, 3)\\
\hline
12&4&8&(-1, 1, -2, 3, -5, 8, -13, -13, 8, 8, -5, 3)\\
\hline
\end{tabular}\\ \\ \\
Since it is known that a sequence $(d_1, d_2, d_3, ..., d_n)$ of non-negative integers is a degree sequence of some graph $G$  \emph{if and only if} $\sum\limits_{i+i}^{n} d_i$ is even. It implies that a degree sequence has an even number of \emph{odd entries.} Hence, we know that the $f^\pm_i$-sequence of $J_n(1)$ denoted, $\Bbb F^\pm(J_n(1)), n \in \Bbb N$ has an \emph{even number} of, $-f_{d(v_i)}$ entries. Following from Table 1 the table below depicts the values $f^\pm Z_1(J_n(1)), f^\pm Z_2(J_n(1)), f^\pm Z_3(J_n(1))$ and $f^\pm Z_4(J_n(1))$ for $J_n(1), n \leq 12.$\\ \\
\noindent Table 2. \\
\begin{tabular}{|c|c|c|c|c|c|c|}
\hline
$i\in{\Bbb{N}}$&$d^-(v_i)$&$d^+(v_i)$&$f^\pm Z_1(J_i(1))$&$f^\pm Z_2(J_i(1))$&$f^\pm Z_3(J_i(1))$&$f^\pm Z_4(J_i(1))$\\
\hline
1&0&1&$0$&0&0&0\\
\hline
2&1&1&$2$&1&0&0\\
\hline
3&1&2&$3$&-2&4&4\\
\hline
4&1&3&$4$&-1&4&8\\
\hline
5&2&3&$8$&-6&11&16\\
\hline
6&2&4&$15$&5&11&25\\
\hline
7&3&4&$32$&-26&35&56\\
\hline
8&3&5&$62$&-19&50&98\\
\hline
9&3&6&$103$&0&72&138\\
\hline
10&4&6&$211$&38&119&251\\
\hline
11&4&7&$396$&-238&210&402\\
\hline
12&4&8&$604$&-158&273&566\\
\hline
\end{tabular} 
\section{Khazamula irregularity}
Let $G^\rightarrow$ be a simple directed graph on $n \geq 2$ vertices labelled $v_1, v_2, v_3, ..., v_n.$  Let all vertices $v_i$ carry its $\pm$\emph{Fibonacci weight}, $f^\pm_i$ related to $d(v_i) = d(v^+(v_i) + d^-(v_i).$ Also let vertex $v_j$ be a head vertex of $v_i$ and choose any $d(v^h_i) = max(d(v_j)_{\forall v_j}).$
\begin{definition}
Let $G^\rightarrow$ be a simple directed graph on $n \geq 2$ vertices with each vertex carrying its $\pm$\emph{Fibonacci weight,} $f^\pm_i.$ For the function $f(x) = mx +c, x \in \Bbb R$ and $m, c \in \Bbb Z$ we define the Khazamula irregularity as:\\ \\
$irr_k(G^\rightarrow) =\sum\limits_{i=1}^{n}|\int_{f^\pm_i}^{d(v^h_i)}f(x)dx|.$
\end{definition}
\noindent Note: Vertices $v$ with $d^+(v) = 0,$ are \emph{headless} and the corresponding integral terms to the summation are defined \emph{zero}. Hence, $irr_k(K_1^\rightarrow) = 0.$\\ \\
\noindent Let $G$ be a simple connected undirected graph on $n$ vertices which are labelled, $v_1, v_2, v_3, ..., v_n.$ Also let $G$ have $\epsilon$ edges. It is known that $G$ can be \emph{orientated} in $2^\epsilon$ ways, including the cases of isomorphism. Finding the relationship between the different values of $irr_k(G^\rightarrow)$ and $irr_k^c(G^\rightarrow)$ (to follow in subsection 3.3) in respect of the different orientations for $G$ in general is stated as an open problem. In this section we give results in respect of particular orientations of paths, cycles, wheels and complete bipartite graphs.
\subsection{$irr_k$ for Paths, Cycles, Wheels and Complete Bipartite Graphs}
\begin{proposition}
For a directed path $P_n^\rightarrow, n \geq 2$ which is consecutively directed from \emph{left to right} we have that the \emph{Khazamula irregularity,} $irr_k(P_n^\rightarrow) = |\frac{3}{2}(n-2)m + nc|.$
\end{proposition}
\begin{proof}
Label the vertices of the directed path $P_n^\rightarrow$ consecutively from left to right $v_1, v_2, v_3, ...., v_n.$ From the definition $irr_k(P^\rightarrow_n) =\sum\limits_{i=1}^{n}|\int_{f^\pm_i}^{d(v^h_i)}f(x)dx|,$ it follows that we have:\\ \\
$\sum\limits_{i=1}^{n}|\int_{f^\pm_i}^{d(v^h_i)}f(x)dx| = |\int^2_{-1}f(x)dx + \underbrace{\int^2_1f(x)dx + \dots+ \int^2_1}_{(n-3)-terms} + \int^1_1f(x)dx|.$\\ \\
So we have, $\sum\limits_{i=1}^{n}|\int_{f^\pm_i}^{d(v^h_i)}f(x)dx| = |(\frac{1}{2}mx^2 + cx)\arrowvert^2_{-1} + (n-3)(\frac{1}{2}mx^2 + cx)\arrowvert^2_1 + 0| =\\
|2m + 2c -\frac{1}{2}m + c + (n-3)(2m + 2c - \frac{1}{2}m - c)| = |\frac{3}{2}m + 3c + \frac{3}{2}(n-3)m + (n-3)c| = |\frac{3}{2}(n-2)m + nc|.$
\end{proof}
\begin{proposition}
For a directed cycle $C_n^\rightarrow$ which is consecutively directed \emph{clockwise} we have that the \emph{Khazamula irregularity,} $irr_k(C_n^\rightarrow) = n|\frac{3}{2}m + c|.$
\end{proposition}
\begin{proof}
Label the vertices of the directed cycle $C_n^\rightarrow$ consecutively clockwise $v_1, v_2, v_3, ...., v_n.$ So vertices carry the $\pm$\emph{Fibonacci weight}, $f^\pm_{i_{\forall i}} = f_1 = 1.$  Also a head vertex is always unique with degree = 2. From the definition $irr_k(C^\rightarrow_n) =\sum\limits_{i=1}^{n}|\int_{f^\pm_i}^{d(v^h_i)}f(x)dx|,$ it follows that we have:\\ \\
$\sum\limits_{i=1}^{n}|\int_{f^\pm_i}^{d(v^h_i)}f(x)dx| =|\underbrace{\int^2_1f(x)dx + \int^2_1f(x)dx + \dots + \int^2_1f(x)dx}_{n-terms}| = |n{(\frac{1}{2}mx^2 + cx)\arrowvert^2_1}| = \\ \\ 
|n(2m + 2c -\frac{1}{2}m - c)| = |n(\frac{3}{2}m + c)| =n|\frac{3}{2}m + c|.$
\end{proof}
\begin{proposition}
For a directed Wheel graph $W^\rightarrow_{(1,n)}$ with the axle vertex $u_1$ and the wheel vertices $v_1, v_2, ..., v_n$ and the spokes directed $(u_1, v_i)_{\forall i}$ and the wheel vertices directed consecutively clockwise $v_1, v_2, ..., v_n,$ we have that:
\begin{equation*} 
irr_k(W_{(1, n)}^\rightarrow)
\begin{cases}
= |\frac{(5n -f_n^2 + 9)}{2}m + (5n - f_n +3)c| , &\text {if n is even,}\\ \\ 
= |\frac{(5n -f_n^2 + 9)}{2}m + (5n + f_n +3)c|, &\text {if n is uneven.}
\end{cases}
\end{equation*} 
\end{proposition}
\begin{proof}
Consider a Wheel graph $W^\rightarrow_{(1,n)}$ with the axle vertex $u_1$ and the wheel vertices $v_1, v_2, ..., v_n$ and the spokes directed $(u_1, v_i)_{\forall i}$ and the wheel vertices directed consecutively clockwise $v_1, v_2, ..., v_n.$ \\ \\
\noindent Case 1: If $n$ is even then $d(u_1)$ is even and carries the $\pm$\emph{Fibonacci weight}, $f_n.$ Obviously the wheel vertices have $d(v_i) = 3_{\forall i},$ hence carry the $\pm$\emph{Fibonacci weight}, $f_3 = -2_{\forall v_i}.$ So from the definition of the Khazamula irregularity we have that:\\ \\
$irr_k(W^\rightarrow_{(1, n)}) =\sum\limits_{i=1}^{n}|\int_{f^\pm_i}^{d(v^h_i)}f(x)dx| = |n\int^3_{-2}f(x)dx + \int^3_{f_n}f(x)dx|$ if $n$ is even. This results in, $irr_k =\sum\limits_{i=1}^{n}|\int_{f^\pm_i}^{d(v^h_i)}f(x)dx| = |n(\frac{9}{2}m + 3c - 2m + 2c) + (\frac{9}{2}m + 3c - \frac{f^2_n}{2}m - f_nc)| = |\frac{5}{2}nm + 5nc + \frac{9}{2}m + 3c - \frac{f^2_n}{2}m - f_nc| = |\frac{(5n - f^2_n +9)}{2}m + (5n - f_n +3)c|.$\\ \\
\noindent Case 2: If $n$ is uneven then $d(u_1)$ is uneven and carries the $\pm$\emph{Fibonacci weight}, $-f_n.$ So in the Riemann integral $\int^3_{-f_n}f(x)dx$ we have $(\frac{9}{2}m + 3c - \frac{f^2_n}{2}m + f_nc).$ So the result $irr_k(W^\rightarrow_{(1, n)}) = |\frac{(5n -f_n^2 + 9)}{2}m + (5n + f_n +3)c|$ if $n$ is uneven, follows.
\end{proof}
\noindent Consider the complete bipartite graph $K_{(n, m)}$ and call the n vertices the \emph{left-side vertices} and the m vertices the \emph{right-side vertices}. Orientate $K_{(n, m)}$ strictly from \emph{left-side vertices} to \emph{right-side vertices} to obtain $K^{l \rightarrow r}_{(n, m)}$.
\begin{proposition}
For the directed graph $K^{l \rightarrow r}_{(n, m)}$ we have that:
\begin{equation*} 
irr_k(K^{l \rightarrow r}_{(n, m)})
\begin{cases}
= |\frac{(n^3 - nf_m^2 )}{2}m + (n^2 - nf_m)c| , &\text {if m is even,}\\ \\ 
= |\frac{(n^3 - nf_m^2)}{2}m + (n^2 + nf_m)c|, &\text {if m is uneven.}
\end{cases}
\end{equation*} 
\end{proposition}
\begin{proof}
For the directed graph $K^{l \rightarrow r}_{(n, m)}$ we have that all \emph{left-side vertices} say $v_1, v_2, ..., v_n$ have $d^+(v_i) = m$, whilst all \emph{right-side vertices} say $u_1, u_2, ..., u_m$ have $d^-(u_i) = n$ and $d^+(u_i) = 0.$\\ \\
\noindent Case 1: If $m$ is even it follows from the definition that, $irr_k(K^{l \rightarrow r}_{(n, m)}) = n|\int^n_{f_m}f(x)dx|.$ So we have that $irr_k(K^{l \rightarrow r}_{(n, m)}) = n|(\frac{1}{2}mx^2 + cx)\arrowvert^n_{f_m}| = n|\frac{n^2}{2}m +nc -(\frac{f^2_m}{2}m +f_mc)| = |\frac{(n^3 - nf_m^2 )}{2}m + (n^2 - nf_m)c|.$\\ \\
\noindent Case 2: If $m$ is uneven the \emph{left-side vertices} all carry the $\pm$\emph{Fibonacci weight}, $-f_m.$ Hence, the result follows as in Case 1, accounting for $-f_m.$
\end{proof}
\noindent \textbf{Example problem 1:} Let $n = 1$ or $5$ and $f(x) = mx.$ Prove that $irr_k(K^\rightarrow_{(1, n)}) = 0$ or $|12m|$ and,
\begin{equation*} 
irr_k(K^{l \rightarrow r}_{(1, n)})
\begin{cases}
= 0\\ \\ 
or\\ \\
= 5 (irr_k(K^\rightarrow_{(1, n)})) = 60|m|.
\end{cases}
\end{equation*} 
\begin{proof}
Let $n=1$ and let $f(x) = mx.$ From the definition of $irr_k(G^\rightarrow)$ it follows that $irr_k(K^\rightarrow_{(1, n)}) =\int^1_{-1}|mx.dx|_{for-v_1} = |\frac{1}{2}mx^2\arrowvert^1_{-1}| = 0.$ We also have that $irr_k(K^\rightarrow_{(1, n)}) =\int^1_{-1}|mx.dx|_{for-u_1}\\ \\ = |\frac{1}{2}mx^2\arrowvert^1_{-1}| = 0.$\\ \\
Let $n=5$ and let $f(x) = mx.$ Now we have that $irr_k(K^\rightarrow_{(1, n)}) =\int^1_{-5}|mx.dx|_{for-v_1} = |\frac{1}{2}mx^2\arrowvert^1_{-5}| = |12m|.$\\ \\
For $irr_k(K^\rightarrow_{(1, n)})$ we have $\sum\limits^5_{i=1}\int^5_{-1}|mx.dx|_{for-u_i, i=1,2,..,5} = 5(\int^5_{-1}|mx.dx|) = 5|\frac{1}{2}mx^2\arrowvert^5_{-1}| = 5|12m| = 60|m|.$ 
\end{proof}
\subsection{Khazamula's Theorem}
Consider two simple connected directed graphs, $G^\rightarrow$ and $H^\rightarrow$. Let the vertices of $G^\rightarrow$ be labelled $v_1, v_2, ..., v_n$ and the vertices of $H^\rightarrow$ be labelled $u_1, u_2, ..., u_m.$ Define the \emph{directed join} as $(G^\rightarrow + H^\rightarrow)^\rightarrow$ conventionally, with the arcs $\{(v_i, u_j)| \forall v_i \in V(G^\rightarrow), u_j \in V(H^\rightarrow\}.$
\begin{theorem}
Consider two simple connected directed graphs, $G^\rightarrow$ on n vertices and $H^\rightarrow$ on m vertices then, $irr_k((G^\rightarrow + H^\rightarrow)^\rightarrow) = |n\int^{\Delta(H^\rightarrow) +n}_{f^\pm _i\arrowvert_{v_i \in V((G^\rightarrow + H^\rightarrow)^\rightarrow)}}f(x)dx + \sum\limits_{i=1}^{m}|\int_{f_{d(u_i)+1}}^{d(u^h_i)+1}f(x)dx|.$
\end{theorem}
\begin{proof}
Note that in the graph $G^\rightarrow$ the maximum degree $\Delta(G^\rightarrow) = max(d^+(v_\ell) + d^-(v_\ell)) \leq n-1$ for at least one vertex $v_\ell.$ If such a vertex $v_\ell$ is indeed the \emph{head vertex} of a vertex $v_t,$ then $\sum\limits_{i=1}^{n}|\int_{f^\pm_i}^{d(v^h_i)}f(x)dx|,$ will contain the term $\int^{\Delta(G)}_{f^\pm _t}f(x)dx.$\\ \\
 In $H^\rightarrow$ the maximum degree $\Delta(H^\rightarrow) = max(d^+(u_s) + d^-(u_s)) \geq 1$ for some vertex $u_s.$ Hence, in the directed graph $(G^\rightarrow + H^\rightarrow)^\rightarrow,$ all terms of $\sum\limits_{i=1}^{n}|\int_{f^\pm_i}^{d(v^h_i)}f(x)dx|$ reduces to \emph{zero} and are replaced by the terms $\int^{\Delta(H^\rightarrow)+ n}_{f^\pm _i\arrowvert_{v_i \in V((G^\rightarrow + H^\rightarrow)^\rightarrow)}}f(x)dx,$ because $\Delta(G^\rightarrow) \leq n-1 < \Delta(H^\rightarrow)+ n.$ \\ \\
In respect of $H^\rightarrow$ we have that each $d(u_i)_{\forall i}$ increases by exactly 1 so the value of $f_{d(u_i)_{\forall i}}$ \emph{switches between} $\pm$ and adopts the value $f_{d(u_i) + 1}.$ Similarly all \emph{head vertices'} degree increases by exactly 1. These observations result in:\\ \\
$irr_k((G^\rightarrow + H^\rightarrow)^\rightarrow) = |n\int^{\Delta(H^\rightarrow) +n}_{f^\pm _i\arrowvert_{v_i \in V((G^\rightarrow + H^\rightarrow)^\rightarrow)}}f(x)dx + \sum\limits_{i=1}^{m}\int_{f_{d(u_i)+1}}^{d(u^h_i)+1}f(x)dx|.$
\end{proof}
\noindent \textbf{Example problem 2:} An application of the Khazamula theorem to the graph $(C_n^\rightarrow + K_1)^\rightarrow$ in respect of $f(x) = mx,$ results in $irr_k((C_n^\rightarrow + K_1)^\rightarrow) = \frac{1}{3}(n^2-4)irr_k(C_n^\rightarrow)_{\arrowvert_{f(x)=mx}}.$
\subsection{Khazamula c-irregularity for orientated Paths, Cycles, Wheels and Complete Bipartite Graphs}
Let $f(x) = \sqrt{r^2 - x^2}, x \in \Bbb R$ and $r = max\{d(v_i)_{\forall v_i, d^-(v_i) \geq 1},$\emph{or} $|(f^\pm_i)|_{\forall v_i}\}.$ We define \emph {Khazamula c-irregularity} as $irr^c_k(G^\rightarrow) = \sum\limits_{i=1}^{n}|\int_{f^\pm_i}^{d(v^h_i)}f(x)dx|.$ It is known that $\int^b_a\sqrt{r^2 - x^2}dx = (\frac{1}{2}x\sqrt{r^2 - x^2} + \frac{r^2}{2}arcsin\frac{x}{r})\arrowvert^b_a.$ Also note that $arcsin\theta$ applies to $\theta \in [-\frac{\pi}{2},\frac{\pi}{2}]$ to ensure a singular value for the respective integral terms. \\ \\
\begin{proposition}
For a directed path $P_n^\rightarrow, n \geq 3$ which is consecutively directed from \emph{left to right} we have that the \emph{Khazamula c-irregularity,} $irr^c_k(P_n^\rightarrow) = (n-2)(\frac{2\pi}{3} - \frac{\sqrt3}{2}).$
\end{proposition}
\begin{proof}
Label the vertices of the directed path $P_n^\rightarrow, n\geq 3$ consecutively from left to right $v_1, v_2, v_3, ...., v_n.$ Note that $r = max\{d(v_i)_{\forall v_i},$\emph{or} $|(f^\pm_i)|_{\forall v_i}\} = 2.$ From the definition $irr^c_k(P^\rightarrow_n) =\sum\limits_{i=1}^{n}|\int_{f^\pm_i}^{d(v^h_i)}f(x)dx|,$ it follows that we have:\\ \\
$\sum\limits_{i=1}^{n}|\int_{f^\pm_i}^{d(v^h_i)}f(x)dx| = |\int^2_{-1}f(x)dx + \underbrace{\int^2_1f(x)dx + \dots+ \int^2_1}_{(n-3)-terms} + \int^1_1f(x)dx|.$\\ \\
So we have, $\sum\limits_{i=1}^{n}|\int_{f^\pm_i}^{d(v^h_i)}f(x)dx| = |(\frac{1}{2}x\sqrt{r^2 - x^2} + \frac{r^2}{2}arcsin\frac{x}{r})\arrowvert^2_{-1} + (n-3)(\frac{1}{2}x\sqrt{r^2 - x^2} + \\ \\ \frac{r^2}{2}arcsin\frac{x}{r})\arrowvert^2_1| = |(\frac{1}{2}x\sqrt{4-x^2} + 2arcsin\frac{x}{2})\arrowvert^2_{-1} +(n-3)(\frac{1}{2}x\sqrt{4-x^2} + 2arcsin\frac{x}{2})\arrowvert^2_1| = |(\frac{2\pi}{3} - \\ \\ \frac{\sqrt3}{2}) + (n-3)(\frac{2\pi}{3} -\frac{\sqrt3}{2})| =|(n-2)(\frac{2\pi}{3} -\frac{\sqrt3}{2})| = (n-2)(\frac{2\pi}{3} -\frac{\sqrt3}{2}).$
\end{proof}
\begin{proposition}
For a directed cycle $C_n^\rightarrow$ which is consecutively directed \emph{clockwise} we have that the \emph{Khazamula c-irregularity,} $irr^c_k(C_n^\rightarrow) = n(\frac{2\pi}{3} - \frac{\sqrt3}{2}).$
\end{proposition}
\begin{proof}
Label the vertices of the directed cycle $C_n^\rightarrow$ consecutively clockwise $v_1, v_2, v_3, ...., v_n.$ So all vertices carry the $\pm$\emph{Fibonacci weight}, $f^\pm_{i_{\forall i}} = f_1 = 1.$  Also a head vertex is always unique with degree = 2. So $r = max\{d(v_i)_{\forall v_i},$\emph{or} $|(f^\pm_i)|_{\forall v_i}\} = 2.$ From the definition $irr_k(C^\rightarrow_n) =\sum\limits_{i=1}^{n}|\int_{f^\pm_i}^{d(v^h_i)}f(x)dx|,$ it follows that we have:\\ \\
$\sum\limits_{i=1}^{n}|\int_{f^\pm_i}^{d(v^h_i)}f(x)dx| =|\underbrace{\int^2_1f(x)dx + \int^2_1f(x)dx + \dots + \int^2_1f(x)dx}_{n-terms}| = n|(\frac{1}{2}x\sqrt{4-x^2} + 2arcsin\frac{x}{2})\arrowvert^2_1| = \\ \\
n|(0 + 2arcsin1 - \frac{\sqrt3}{2} - 2arcsin\frac{1}{2})| = n|(\frac{2\pi}{3} - \frac{\sqrt3}{2})| = n(\frac{2\pi}{3} - \frac{\sqrt3}{2}).$
\end{proof}
\begin{proposition}
For a directed Wheel graph $W^\rightarrow_{(1,n)}$ with the axle vertex $u_1$ and the wheel vertices $v_1, v_2, ..., v_n$ and the spokes directed $(u_1, v_i)_{\forall i}$ and the wheel vertices directed consecutively clockwise $v_1, v_2, ..., v_n,$ we have that:
\begin{equation*} 
irr_k(W_{(1, n)}^\rightarrow)
\begin{cases}
= 4\sqrt5 + 9\pi + 18arcsin\frac{2}{3}, &\text {if n= 3 or 4,}\\ \\
= |\frac{3}{2}(n+1)\sqrt{f^2_n - 9} + (n+1)\frac{f^2_n}{2}arcsin\frac{3}{f_n}-\frac{f^2_n\pi}{4} + A|, &\text {if $n \geq 6$ and even,}\\ \\ 
= |\frac{3}{2}(n+1)\sqrt{f^2_n - 9} + (n+1)\frac{f^2_n}{2}arcsin\frac{3}{f_n}+\frac{f^2_n\pi}{4} + B|, &\text {if $n \geq 5$ and uneven,}
\end{cases}
\end{equation*}
with: $A = n(\sqrt{f^2_n - 4} + \frac{f^2_n}{2}arcsin\frac{2}{f_n})$ and $B = n(\sqrt{f^2_n - 4} - \frac{f^2_n}{2}arcsin\frac{2}{f_n}).$
\end{proposition}
\begin{proof}
Consider a Wheel graph $W^\rightarrow_{(1,n)}$ with the axle vertex $u_1$ and the wheel vertices $v_1, v_2, ..., v_n$ and the spokes directed $(u_1, v_i)_{\forall i}$ and the wheel vertices directed consecutively clockwise $v_1, v_2, ..., v_n.$ \\ \\
\noindent Case 1: If $n = 3$ we have that $irr^c_k(W^\rightarrow_{(1, 3)}) = |\underbrace{\int^3_{-2}\sqrt{9 - x^2}dx}_{for, u_1} + 3\underbrace{\int^3_{-2}\sqrt{9 - x^2}dx}_{for, v_i}|, i= 1, 2, 3.$\\ \\
Therefore, $irr^c_k(W^\rightarrow_{(1, 3)}) = 4|(\int^3_{-2}\sqrt{9 - x^2}dx)| = 4|(\frac{1}{2}x\sqrt{9 - x^2} + \frac{9}{2}arcsin\frac{x}{3})\arrowvert^3_{-2}| = 4|(\frac{9}{2}arcsin1 - \\ \\ (-\sqrt{9 - 4}-\frac{9}{2}arcsin\frac{2}{3}))| = 4\sqrt5 + 9\pi +18arcsin\frac{2}{3}.$\\ \\
If $n = 4$ then $irr^c_k(W^\rightarrow_{(1, 4)}) = |\underbrace{\int^3_{3}\sqrt{9 - x^2}dx}_{for, u_1} + 4\underbrace{\int^3_{-2}\sqrt{9 - x^2}dx}_{for, v_i}|, i= 1, 2, 3, 4.$ Hence, the result follows.\\ \\
\noindent Case 2: If $n \geq 6$ and \emph{even} we have $irr^c_k(W^\rightarrow_{(1, n)}) = |\underbrace{\int^3_{f_n}\sqrt{f^2_n - x^2}dx}_{for, u_1} +  n\underbrace{\int^3_{-2}\sqrt{f^2_n - x^2}dx}_{for,v_i}|, i =\\ \\ 1, 2, ..., n.$ So we have $irr^c_k(W^\rightarrow_{(1, n)}) = |(\frac{1}{2}x\sqrt{f_n^2 - x^2} + \frac{f_n^2}{2}arcsin\frac{x}{f_n})\arrowvert^3_{f_n} + n(\frac{1}{2}x\sqrt{f^2_n - x^2} +\\ \\ \frac{f^2_n}{2}arcsin\frac{x}{f_n})\arrowvert^3_{-2})| = |\frac{3}{2}\sqrt{f^2_n - 9} +\frac{f^2_n}{2}arcsin\frac{3}{f_n}-(\frac{f_n}{2}\sqrt{f^2_n - f^2_n} + \frac{f^2_n}{2}arcsin1) + n(\frac{3}{2}\sqrt{f^2_n - 9} + \\ \\\frac{f^2_n}{2}arcsin\frac{3}{f_n} -(-\sqrt{f^2_n - 4} -\frac{f^2_n}{2}arcsin\frac{2}{f_n}))|=|\frac{3}{2}\sqrt{f^2_n - 9} +\frac{f^2_n}{2}arcsin\frac{3}{f_n} - (\frac{f_n}{2}\sqrt{f^2_n - f^2_n} + \\ \\ \frac{f^2_n}{2}arcsin1) + n(\frac{3}{2}\sqrt{f^2_n - 9} +\frac{f^2_n}{2}arcsin\frac{3}{f_n} +\sqrt{f^2_n - 4} + \frac{f^2_n}{2}arcsin\frac{2}{f_n})| = |\frac{3}{2}(n+1)\sqrt{f^2_n - 9} + \\ \\ (n+1)\frac{f^2_n}{2}arcsin\frac{3}{f_n} + \frac{f^2_n\pi}{4} + A|$, with $A = n(\sqrt{f^2_n - 4} +\frac{f^2_n}{2}arcsin\frac{2}{f_n}).$\\ \\
\noindent Case 3: Similar to Case 2 and accounting for $n \geq 5$ \emph{and uneven}.
\end{proof}
\noindent Consider the complete bipartite graph $K_{(n, m)}$ and call the n vertices the \emph{left-side vertices} and the m vertices the \emph{right-side vertices}. Orientate $K_{(n, m)}$ strictly from \emph{left-side vertices} to \emph{right-side vertices} to obtain $K^{l \rightarrow r}_{(n, m)}$.
\begin{proposition}
For the directed graph $K^{l \rightarrow r}_{(n, m)}$ we have that:
\begin{equation*} 
irr^c_k(K^{l \rightarrow r}_{(n, m)})
\begin{cases}
= |\frac{n^2\pi}{4} - A|, &\text {if $n \geq f_m$ and m is even,}\\ \\ 
= |\frac{n^2\pi}{4} + A|, &\text {if $n \geq f_m$ and m is uneven,}\\ \\ 
= |B - \frac{f^2_m\pi}{4}|, &\text {if $f_m > n$ and m is even,}\\ \\
= |B + \frac{f^2_m\pi}{4}|, &\text {if $f_m > n$ and m is uneven,}
\end{cases}
\end{equation*} 
\noindent with $A = \frac{f_m}{2}\sqrt{n^2 - f^2_m} + \frac{n^2}{2}arcsin\frac{f_m}{n}$ and $B = \frac{n}{2}\sqrt{f_m^2 - n^2} + \frac{f_m^2}{2}arcsin\frac{n}{f_m}.$
\end{proposition}
\begin{proof}
For the directed graph $K^{l \rightarrow r}_{(n, m)}$ we have that all \emph{left-side vertices} say $v_1, v_2, ..., v_n$ have $d^+(v_i) = m$, whilst all \emph{right-side vertices} say $u_1, u_2, ..., u_m$ have $d^-(u_i) = n$ and $d^+(u_i) = 0.$\\ \\
\noindent Case 1: Since $d^+(u_i) = 0, \forall i$ the terms in $\sum\limits_{i=1}^{n}|\int_{f^\pm_i}^{d(v^h_i)}f(x)dx|,$ stem from vertices $v_i, \forall i$ only. Furthermore, since $r = max\{d(u_i)_{\forall i, d^-(u_i)\geq1},$\emph{or} $f_m\}$ and $n \geq f_m,$ we have $r=n.$\\ \\
It follows that $irr^c_k(K^{l \rightarrow r}_{(n, m)}) = n|\int^n_{f_m}\sqrt{n^2 - x^2}dx| = |(\frac{1}{2}x\sqrt{n^2 - x^2} + \frac{n^2}{2}arcsin\frac{x}{n})\arrowvert^n_{f_m}| =\\ \\
|\frac{n^2}{2}arcsin1 - (\frac{f_m}{2}\sqrt{n^2 - f^2_m} + \frac{n^2}{2}arcsin\frac{f_m}{n})| = |\frac{n^2\pi}{4} - A|,$ with $A= \frac{f_m}{2}\sqrt{n^2 - f^2_m} + \frac{n^2}{2}arcsin\frac{f_m}{n}.$\\ \\
\noindent Case 2: Similar to Case 1 and accounting for \emph{m is uneven}.\\ \\
\noindent Case 3: Similar to Case 1 and accounting for $f_m > n,$ \emph{m is even}.\\ \\
\noindent Case 4:  Similar to Case 1 and accounting for $f_m > n,$ \emph{m is uneven}.
\end{proof}
\noindent [Open problem: If possible, generalise Khazamula's irregularity for simple directed graphs.] \\  \\
\noindent [Open problem: Find a closed or, recursive formula for $f^\pm Z_1(J_n(1)), f^\pm Z_2(J_N(1)), f^\pm Z_3(J_n(1)),$\\ and $f^\pm Z_4(J_n(1)).$]\\ \\
\noindent [Open problem: Where possible, describe the terms of the Khazamula theorem in terms of $irr_k(G^\rightarrow)$ and $irr_k(H^\rightarrow)$ for specialised classes of simple directed graphs.]\\ \\
\noindent [Open problem: If possible, formulate and prove \emph{Khazamula's c-Theorem} related to \emph{Khazamula c-irregularity} for simple directed graphs in general.] \\ \\
\noindent [Open problem: Let $G$ be a simple connected undirected graph on $n$ vertices labelled, $v_1, v_2, v_3, ..., v_n.$ Also let $G$ have $\epsilon$ edges. It is known that $G$ can be \emph{orientated} in $2^\epsilon$ ways, including the cases of isomorphism. Find the relationship between the different values of $irr_k(G^\rightarrow)$ in respect of the different orientations.]\\ \\
\noindent [Open problem: Let $G$ be a simple connected undirected graph on $n$ vertices labelled, $v_1, v_2, v_3, ..., v_n.$ Also let $G$ have $\epsilon$ edges. It is known that $G$ can be \emph{orientated} in $2^\epsilon$ ways, including the cases of isomorphism. Find the relationship between the different values of $irr_k^c(G^\rightarrow)$ in respect of the different orientations.]\\ \\
\noindent \textbf{\emph{Open access:\footnote {To be submitted to the \emph{Pioneer Journal of Mathematics and Mathematical Sciences.}}}} This paper is distributed under the terms of the Creative Commons Attribution License which permits any use, distribution and reproduction in any medium, provided the original author(s) and the source are credited. \\ \\
References (Limited) \\ \\
$[1]$ Abdo, H., Dimitrov, D., \emph{The total irregularity of a graph,} arXiv: 1207.5267v1 [math.CO], 24 July 2012. \\  
$[2]$ Alavi, Y., Boals, A., Chartrand, G., Erd\"os, P., Graham, R., Oellerman, O., \emph{k-path Irregular Graphs}, Congressus Numerantium, Vol 65, (1988), pp 201-210.\\
$[3]$ Albertson, M.O., \emph{The irregularity of a graph}, Ars Combinatoria, Vol 46 (1997), pp 219-225.\\
$[4]$ Ashrafi, A.R., Do\v sli\'c, T., Hamzeha, A., \emph{The Zagreb coindices of graph operations}, Discrete Applied Mathematics, Vol 158 (2010), pp 1571-1578\\
$[5]$ Bondy, J.A., Murty, U.S.R., \emph {Graph Theory with Applications,} Macmillan Press, London, (1976). \\
$[6]$ Cvetkovi\' c, D., Rowlinson,P., \emph{On connected graphs with maximal index,} Publications de I'Institut Mathematique Beograd, Vol 44 (1988), pp 29-34.\\
$[7]$ Fath-Tabar, G.H., \emph{Old and new Zagreb indices of graphs}, MATCH Communications in Mathematical and in Computer Chemistry, Vol 65 (2011), pp 79-84.\\
$[8]$ Henning, M.A., Rautenbach, D., \emph{On the irregularity of bipartite graphs}, Discrete Mathematics, Vol 307 (2007), pp 1467-1472.\\
$[9]$ Kok, J., Fisher, P., Wilkens, B., Mabula, M., Mukungunugwa, V., \emph{Characteristics of Finite Jaco Graphs, $J_n(1), n \in \Bbb N$}, arXiv: 1404.0484v1 [math.CO], 2 April 2014. \\
$[10]$  Kok, J., Fisher, P., Wilkens, B., Mabula, M., Mukungunugwa, V., \emph{Characteristics of Jaco Graphs, $J_\infty(a), a \in \Bbb N$}, arXiv: 1404.1714v1 [math.CO], 7 April 2014. \\
$[11]$ Kok, J., \emph{Total irregularity and $f_t$-irregularity of Jaco Graphs, $J_n(1), n \in \Bbb N$}, arXiv: 1406.6168v1 [math.CO], 24 June 2014. \\
$[12]$ Kok, J., \emph{Total Irregularity of Graphs resulting from Edge-joint, Edge-transformation and $irr^+_t(G)$ and $irr^-_t(G)$ of Directed Graphs}, arXiv: 1406.6863v2 [math.CO], 26 June 2014. \\
$[13]$ Zhu, Y., You, L., Yang, J., \emph{The Minimal Total Irregularity of Graphs,} arXiv: 1404.0931v1 [math.CO], 3 April 2014. \\ 
\end{document}